\documentclass[10pt,oneside]{article}
 \usepackage[a4paper,centering,bindingoffset=0cm,hmargin=3cm,vmargin=3cm,includeheadfoot]{geometry}
\usepackage[font=footnotesize,labelfont=bf,width=0.75\textwidth,labelsep=period]{caption}
\usepackage{amsmath}
\usepackage{amssymb,bbm,color}
\usepackage{graphicx}
\usepackage{authblk}
\usepackage{bm}

  \allowdisplaybreaks[1]
   \usepackage{enumerate}

\usepackage[amsmath,thmmarks]{ntheorem}
\usepackage{aliascnt}

\let\RE\Re
\let\Re=\undefined
\DeclareMathOperator{\Re}{\RE e}
\let\IM\Im
\let\Im=\undefined
\DeclareMathOperator{\Im}{\IM m}
\newcommand{\R}{\mathbbm R}

\newcommand{\C}{\mathbbm C}

\newcommand{\SQ}{\mathbbm S}

\renewcommand{\i}{\mathrm i}

\newcommand{\Di}{\Omega_{int}}
\newcommand{\De}{\Omega_{ext}}
\newcommand{\Hi}{\textbf H^{int}}
\newcommand{\Ei}{\textbf E^{int}}
\newcommand{\He}{\textbf H^{ext}}
\newcommand{\Ee}{\textbf E^{ext}}

\newcommand{\Hin}{\textbf H^{inc}}
\newcommand{\Ein}{\textbf E^{inc}}
\newcommand{\hin}{\textbf h^{inc}}
\newcommand{\ein}{\textbf e^{inc}}

\newcommand{\n}{\bm{\hat{n}}}
\newcommand{\di}{\bm{\hat{d}}}
\newcommand{\p}{\bm{\hat{p}}}
\newcommand{\cee}{e_3^{ext}}
\newcommand{\che}{h_3^{ext}}
\newcommand{\cei}{e_3^{int}}
\newcommand{\chii}{h_3^{int}}
\newcommand{\ta}{\bm{\hat{\tau}}}
\newcommand{\bb}[1]{\bm{\mathcal{#1}}}

\DeclareMathOperator{\curl}{\nabla \times}

\newaliascnt{proposition}{lemma}

\aliascntresetthe{proposition}

\newaliascnt{theorem}{lemma}
\newtheorem{theorem}[theorem]{Theorem}
\aliascntresetthe{theorem}

\newaliascnt{iterative}{lemma}
\newtheorem{iterative}[iterative]{Iterative Scheme}
\aliascntresetthe{iterative}

\newaliascnt{assumption}{lemma}

\aliascntresetthe{assumption}

\newaliascnt{remark}{lemma}
\newtheorem{remark}[remark]{Remark}
\aliascntresetthe{remark}

\theoremstyle{nonumberplain}
\theoremseparator{:}
\theoremheaderfont{\normalfont\itshape}
\theorembodyfont{\normalfont}
\theoremsymbol{\ensuremath{\square}}
\newtheorem{proof}{Proof}

\makeatletter
\newcommand{\logmessage}[1]{\@latex@warning{#1}}
\makeatother

\providecommand{\keywords}[1]{\small\textbf{Keywords} #1}

\usepackage{float}

\begin{document}

\title{The inverse electromagnetic scattering problem by a penetrable cylinder at oblique incidence}

\author[1]{Drossos Gintides\thanks{dgindi@math.ntua.gr}}
\author[2]{Leonidas Mindrinos\thanks{leonidas.mindrinos@univie.ac.at}}
\affil[1]{\small Department of Mathematics, National Technical University of Athens, Greece.}
\affil[2]{Computational Science Center, University of Vienna,  Austria.}

\renewcommand\Authands{ and }
\normalsize

\date{}

\maketitle

\begin{abstract}
In this work we consider the method of non-linear boundary integral equation for solving numerically the inverse scattering problem of obliquely incident electromagnetic waves by a penetrable homogeneous cylinder in three dimensions. We consider the indirect method and simple representations for the electric and the magnetic fields in order to derive a system of five integral equations, four on the boundary of the cylinder and one on the unit circle where we measure the far-field pattern of the scattered wave. We solve the system iteratively by linearizing only the far-field equation. Numerical results illustrate the feasibility of the proposed scheme.

\vspace{0.2cm}
\keywords{inverse electromagnetic scattering, oblique incidence, integral equation method}
\end{abstract}

\section{Introduction}

The inverse obstacle scattering problem is to image the scattering object, i.e. find its shape and location, from the knowledge of the far-field pattern of the scattered wave. The medium is illuminated by light at given direction and polarization. Then, Maxwell's equations are used to model the propagation of the light through the medium, see \cite{ColKre13, ColKre14} for an overview. Due to the complexity of the combined system of equations for the electric and the magnetic fields, it is common to impose additional assumptions on the incident illumination and the nature of the scatterer.   

We consider time-harmonic incident electromagnetic plane wave that due to the linearity of the problem will result to a time-independent system of equations. In addition, the penetrable object is considered as an infinitely long homogeneous cylinder. Then, it is characterized by constant permittivity and permeability. The problem is further simplified if we impose oblique incidence for the incident wave. 

The three-dimensional scattering problem modeled by Maxwell's equations is then equivalent to a pair of two-dimensional Helmholtz equations for two scalar fields (the third components of the electric and the magnetic fields). This approach reduces the difficulty of the problem but results to more complicated boundary conditions. The transmission conditions now contain also the tangential derivatives of the electric and magnetic fields. In \cite{GinMin16} we showed that the corresponding direct problem is well-posed and we constructed a unique solution using the direct integral equation method. A similar problem has been considered for an impedance cylinder embedded in a homogeneous \cite{WanNak12}, and in an inhomogeneous medium \cite{NakWan13}. A numerical solution of the direct problem has been also proposed using the finite element method \cite{CanLee91}, the Galerkin method \cite{LucPanSche10}, and the method of auxiliary sources \cite{TsiAliAnaKak07}.

On the other hand, the inverse problem is non-linear and ill-posed. The non-linearity is due to the dependence of the solution of the scattering problem on the unknown boundary curve. The smoothness of the mapping from the boundary to the far-field pattern reflects the ill-posedness of the inverse problem. The unique solvability of the inverse problem is still an open problem. The first and only, to our knowledge, uniqueness result was presented recently in \cite{NakSleWan12} for the case of an impedance cylinder using the Lax-Phillips method.

In this work, we solve the inverse problem by formulating an equivalent system of non-linear integral equations that is solved using a regularized iterative scheme. This method was introduced by Kress and Rundell \cite{KreRun05} and then considered in many different problems, in acoustic scattering problems \cite{IvaJoh07, IvaJoh08}, in elasticity \cite{ChaGinMin16, GinMin11} and in electrical impedance problem \cite{EckKre07}. We propose an iterative scheme that is based on the idea of Johansson and Sleeman \cite{JohSle07} applied to the inverse problem of recovering a perfect conducting cylinder. See \cite{AltKre12a, ChaGinMin16} for some recent applications. We assume integral representations for the solutions that results to a system consisting of four integral equations on the unknown boundary (considering the transmission conditions) and one on the unit circle (taking into account the asymptotic expansion of the solutions). 

We solve this system in two steps. First, given an initial guess for the boundary curve we solve the well-posed subsystem (equations on the boundary) to obtain the corresponding densities and then we solve the linearized (with respect to the boundary) ill-posed far-field equation to update the initial approximation of the radial function. We consider Tikhonov regularization and the normal equations are solved by the conjugate gradient method.

The paper is organized as follows: in Section~\ref{direct} we present the direct scattering problem, the elastic potentials and the equivalent system of integral equations that provide us with the far-field data. 
The inverse problem is stated in Section~\ref{inverse} where we construct an equivalent system of integral equation using the indirect integral equation method. In Section~\ref{numerics} the two-step method for the parametrized form of the system  and the necessary Fr\'echet derivative of the integral operators are presented. The numerical examples give satisfactory results and justify the applicability of the proposed iterative scheme.

 \section{The direct problem}\label{direct}
 
   We consider the scattering of an electromagnetic wave by a penetrable cylinder in $\R^3$. Let $\textbf x  = (x,y,z) \in \R^3 .$ We denote by $\Di = \{ \textbf x : (x,y) \in \Omega , z \in \R \}$ the cylinder,  where $\Omega$ is a bounded domain in $\R^2$ with smooth boundary $\Gamma .$ The cylinder $\Di$ is oriented parallel to the $z$-axis and $\Omega$ is its horizontal cross section. We assume constant permittivity $\epsilon_0$ and permeability $\mu_0$  for the exterior domain $\De : = \R^3 \setminus \overline{\Omega}_{int}.$ The interior domain  $\Di$ is also characterized by constant parameters $\epsilon_1$ and $\mu_1 .$

We define the exterior magnetic $\He (\textbf x , t)$ and electric field $\Ee (\textbf x , t)$
for $\textbf x \in \De, \, t\in \R$ and the interior fields  $\Hi (\textbf x , t)$ and $\Ei (\textbf x , t)$ for $\textbf x \in \Di, \, t\in \R ,$ that satisfy the Maxwell's equations
\begin{equation}\label{eq_Maxwell}
  \begin{aligned}
\curl \Ee + \mu_0 \frac{\partial \He}{\partial t} &= 0, & \curl \He - \epsilon_0 \frac{\partial \Ee}{\partial t} &= 0,  & \textbf x \in \De ,\\
\curl \Ei + \mu_1 \frac{\partial \Hi}{\partial t} &= 0, & \curl \Hi - \epsilon_1 \frac{\partial \Ei}{\partial t} &= 0,  
& \textbf x \in \Di . 
\end{aligned}
\end{equation}
and the transmission conditions
\begin{equation}\label{bound_cond}
\n \times \Ei  =   \n \times \Ee, \quad \n \times \Hi  =   \n \times \He,  \quad \textbf x\in \Gamma ,
\end{equation}
where $\n$ is the outward normal vector, directed into $\De$. 

We illuminate the cylinder with an incident electromagnetic plane wave at oblique incidence, meaning transverse magnetic (TM) polarized wave. We define by $\theta$  the incident angle with respect to the negative $z$ axis and by $\phi$ the polar angle of the incident direction $\di$ (in spherical coordinates), see Figure~\ref{Fig1}. Then, $\di = (\sin \theta \cos \phi , \sin \theta \sin \phi , -\cos \theta )$ and the polarization vector is given by $\p = (\cos \theta \cos \phi , \cos \theta \sin \phi , \sin \theta ), $ satisfying $\di \perp \p$ and assuming that $\theta \in (0,\pi /2) \cup (\pi/2 , \pi). $ 

\begin{figure}
\begin{center}
\includegraphics[scale=0.6]{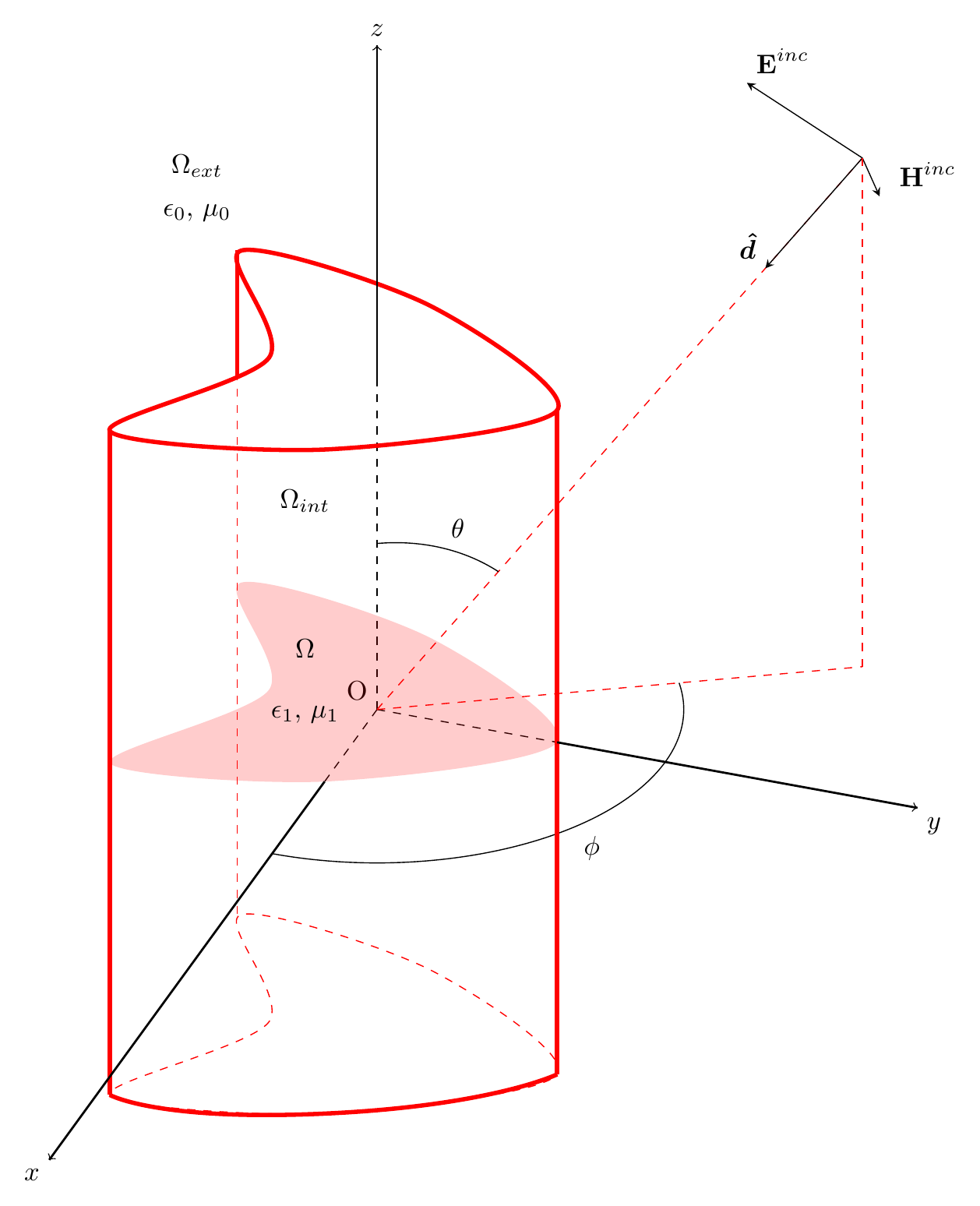}
\caption{The geometry of the scattering problem.}\label{Fig1}
\end{center}
\end{figure}

In the following, due to the linearity of the problem we suppress the time-dependence of the fields and because of the cylindrical symmetry of the medium we express the incident fields as separable functions of $\bm x := (x,y)$ and $z.$ 

Let $\omega >0$ be the frequency and $k_0 = \omega \sqrt{\mu_0 \epsilon_0}$  the wave number in $\De$. We define $\beta = k_0 \cos \theta$ and $\kappa_0 = \sqrt{k_0^2 - \beta^2} = k_0 \sin \theta$ and it follows that the incident fields can be decomposed to \cite{GinMin16}
\begin{equation}\label{eq_incident3}
\begin{aligned}
\Ein (\textbf x ;\di , \p ) = \ein (\bm x) \, e^{-i\beta z}, \quad
\Hin (\textbf x  ;\di , \p ) = \hin (\bm x) \, e^{-i\beta z},
\end{aligned} 
\end{equation}
where
\begin{align*}
\ein (\bm x) &= \frac1{\sqrt{\epsilon_0}} \, \p  \, e^{i\kappa_0 (x \cos \phi + y \sin \phi )}, \quad
\hin (\bm x) = \frac1{\sqrt{\mu_0}} \, (\sin \phi , -\cos \phi ,0) \, e^{i\kappa_0 (x \cos \phi + y \sin \phi )}.
\end{align*}

After some calculations, we can reformulate Maxwell's euqations \eqref{eq_Maxwell} as a system of equations only for the $z$-component of the electric and magnetic fields \cite{GinMin16}. The interior fields $\cei (\bm x)$ and $\chii (\bm x), \ \bm x \in \Omega_1 : = \Omega$ and the exterior fields $\cee (\bm x)$ and $\che (\bm x),$ $\bm x \in \Omega_0  : = \R^2 \setminus \Omega$ satisfy the Helmholtz equations 
\begin{equation}\label{maxwell_third}
\begin{aligned}
\Delta \cei + \kappa^2_1 \,\cei = 0, \quad 
\Delta \chii + \kappa^2_1 \,\chii = 0, \quad \bm x &\in \Omega_1 , \\
\Delta \cee + \kappa^2_0 \,\cee = 0, \quad 
\Delta \che + \kappa^2_0 \,\che = 0, \quad \bm x &\in  \Omega_0 ,
\end{aligned}
\end{equation}
where $\kappa^2_1 = \mu_1 \epsilon_1 \omega^2 - \beta^2 .$ Here, we assume $\mu_1 \epsilon_1 >  \mu_0 \epsilon_0\cos^2\theta$ in order to have $\kappa^2_1 >0.$

The transmission conditions \eqref{bound_cond} can also be written only for the $z$-component of the fields. Let $(\n , \ta)$ be a local coordinate system, where $\n = (n_1 , n_2)$ is the outward normal vector and $\ta = (-n_2 , n_1)$ the outward tangent vector on $\Gamma .$ 
We define $\tfrac{\partial}{\partial n } = \n \cdot \nabla_t,$ $\tfrac{\partial}{\partial \tau } = \ta \cdot \nabla_t ,$ where $\nabla_t = \textbf e_1 \tfrac{\partial}{\partial x}+\textbf e_2 \frac{\partial}{\partial y}$ and $\textbf e_1, \textbf e_2$ denote the unit vectors in $\R^2 .$ Then, we rewrite the boundary conditions as \cite{GinMin16}
\begin{equation}\label{boundary_third}
\begin{aligned}
\cei &= \cee , & \bm x\in \Gamma , \\
\tilde\mu_1 \omega \frac{\partial \chii}{\partial n }  + \beta_1 \frac{\partial \cei}{\partial \tau } &= \tilde\mu_0 \omega \frac{\partial \che}{\partial n }  + \beta_0 \frac{\partial \cee}{\partial \tau }, & \bm x\in \Gamma , \\
\chii &= \che , & \bm x\in \Gamma , \\
\tilde\epsilon_1 \omega \frac{\partial \cei}{\partial n }  - \beta_1 \frac{\partial \chii}{\partial \tau } &= \tilde\epsilon_0 \omega \frac{\partial \cee}{\partial n }  - \beta_0 \frac{\partial \che}{\partial \tau }, & \bm x\in \Gamma ,
\end{aligned}
\end{equation}
where $\tilde\mu_j = \mu_j / \kappa_j^2 , \, \tilde\epsilon_j = \epsilon_j / \kappa_j^2 , \, \beta_j = \beta / \kappa_j^2 ,$ for $j=0,1.$ The exterior fields are decomposed to $\cee = e_3^{sc} + e_3^{inc}$ and $\che = h_3^{sc} + h_3^{inc},$ where $e_3^{sc}$ and $h_3^{sc}$ denote the scattered electric and magnetic field, respectively. From \eqref{eq_incident3} we see that 
\begin{equation}\label{incident_el}
\begin{aligned}
e_3^{inc} (\bm x) = \frac1{\sqrt{\epsilon_0}} \sin \theta \, e^{i\kappa_0 (x \cos \phi +y \sin \phi )}, \quad
h_3^{inc} (\bm x) = 0.
\end{aligned}
\end{equation}

To ensure that the scattered fields are outgoing, we impose in addition the radiation conditions in $\R^2:$
\begin{equation}\label{radiation}
\begin{aligned}
\lim_{r \rightarrow \infty} \sqrt{r} \left( \frac{\partial e_3^{sc}}{\partial r} - i\kappa_0 e_3^{sc} \right) =0 , \quad
 \lim_{r \rightarrow \infty} \sqrt{r} \left( \frac{\partial h_3^{sc}}{\partial r} - i\kappa_0 h_3^{sc} \right) =0 , \\
\end{aligned}
\end{equation}   
where $r = |\bm x |,$ uniformly over all directions.

Now we are in position to formulate the direct transmission problem for oblique incident wave: Find the fields $\chii , h_3^{sc}, \cei$ and $e_3^{sc}$ that satisfy the Helmholtz equations \eqref{maxwell_third}, the transmission conditions \eqref{boundary_third} and the radiation conditions \eqref{radiation}.

\begin{theorem}\label{theo32}
If $\kappa_1^2$ is not an interior Dirichlet eigenvalue and $\kappa_0^2$ is not an interior Dirichlet and Neumann eigenvalue, then the direct transmission problem \eqref{maxwell_third} -- \eqref{radiation} admits a unique solution.
\end{theorem}

\begin{proof}
The proof is based on the integral representation of the solution resulting to a Fredholm type system of boundary integral equations. For more details see \cite[Theorem 3.2]{GinMin16}.
\end{proof}

In the following, $j = 0,1$ counts for the exterior ($\bm x\in \Omega_0$) and interior domain ($\bm x\in \Omega_1$), respectively. We introduce the single- and double-layer potentials defined by
\begin{equation}\label{single_double}
\begin{aligned}
(\mathcal S_j f) (\bm x) &= \int_\Gamma \Phi_j (\bm x,\bm y) f(\bm y) ds (\bm y), & \bm x \in\Omega_j , \\
(\mathcal D_j f) (\bm x) &= \int_\Gamma \frac{\partial \Phi_j}{\partial n (\bm y)} (\bm x,\bm y) f(\bm y) ds (\bm y), & \bm x \in\Omega_j ,
\end{aligned}
\end{equation}
where $\Phi_j$ is the fundamental solution of the Helmholtz equation in $\R^2:$
\begin{equation}
\Phi_j (\bm x,\bm y) = \frac{i}4 H_0^{(1)} (\kappa_j |\bm x-\bm y|), \quad \bm x,\bm y \in\Omega_j , \quad \bm x \neq \bm y,  
\end{equation}
and $H_0^{(1)}$ is the Hankel function of the first kind and zero order. We define also the integral operators
\begin{equation}\label{operators}
\begin{aligned}
( S_j f) (\bm x) &= \int_\Gamma \Phi_j (\bm x,\bm y) f(\bm y) ds (\bm y), & \bm x \in\Gamma , \\
( D_j f) (\bm x) &= \int_\Gamma \frac{\partial \Phi_j}{\partial n (\bm y)} (\bm x,\bm y) f(\bm y) ds (\bm y), & \bm x \in\Gamma , \\
 (NS_j f) (\bm x) &= \int_\Gamma \frac{\partial \Phi_j}{\partial n (\bm x)} (\bm x,\bm y) f(\bm y) ds (\bm y)   , & \bm x \in\Gamma, \\
(ND_j f) (\bm x)  &= \int_\Gamma \frac{\partial^2 \Phi_j}{\partial n (\bm x) \partial n (\bm y)} (\bm x,\bm y) f(\bm y) ds (\bm y), & \bm x \in\Gamma, \\
(TS_j f) (\bm x) &= \int_\Gamma \frac{\partial \Phi_j}{\partial \tau (\bm x)} (\bm x,\bm y) f(\bm y) ds (\bm y) , & \bm x \in\Gamma, \\
(TD_j f) (\bm x) &= \int_\Gamma \frac{\partial^2 \Phi_j}{\partial \tau (\bm x) \partial n (\bm y)} (\bm x,\bm y) f(\bm y) ds (\bm y) , & \bm x \in\Gamma .
\end{aligned}
\end{equation}

The following theorem was proven in \cite{GinMin16}.
\begin{theorem}
Let the assumptions of Theorem~\ref{theo32} still hold. Then, the potentials
\begin{equation}\label{potentials}
\begin{aligned}
\cei (\bm x) &= -(\mathcal D_1 \phi_1 ) (\bm x) + (\mathcal S_1 \eta_1 ) (\bm x), & \bm x\in \Omega_1 , \\
\chii (\bm x) &= -(\mathcal D_1 \psi_1 ) (\bm x) + (\mathcal S_1 \xi_1 ) (\bm x), & \bm x\in \Omega_1 , \\
\cee (\bm x) &= (\mathcal D_0 \phi_0 ) (\bm x) - (\mathcal S_0 \eta_0 ) (\bm x), & \bm x\in \Omega_0 , \\
\che  (\bm x) &= (\mathcal D_0 \psi_0 ) (\bm x) - (\mathcal S_0 \xi_0 ) (\bm x), & \bm x\in \Omega_0 ,
\end{aligned}
\end{equation}
solve the direct transmission problem \eqref{maxwell_third} -- \eqref{radiation} provided that the densities $\phi_0 \in H^{1/2} (\Gamma)$ and $\psi_0 \in H^{1/2} (\Gamma)$ satisfy the system of integral equations
\begin{equation}\label{eq_general1}
(\textbf D_0 + \textbf K_0 ) \begin{pmatrix}
\phi_0 \\ \psi_0
\end{pmatrix} = \textbf b_0 ,
\end{equation}
where
\begin{align*}
\textbf D_0 &= \begin{pmatrix}
D_0 -\tfrac12 I & 0 \\ 0 & D_0 -\tfrac12 I
\end{pmatrix},  &
\textbf K_0 &= \begin{pmatrix}
- \tfrac{\tilde\epsilon_1}{\tilde\epsilon_0} S_0 K_1 & - \frac1{\tilde\epsilon_0 \omega} S_0 (\beta_1 L_1 + \beta_0 L_0 )  \\ \frac{1}{\tilde\mu_0 \omega} S_0 (\beta_1 L_1 + \beta_0 L_0 )  & - \frac{\tilde\mu_1}{\tilde\mu_0} S_0 K_1
\end{pmatrix}, \\
\textbf b_0 &= \begin{pmatrix}
- S_0 \partial_\eta  + \frac{\tilde\epsilon_1}{\tilde\epsilon_0} S_0 K_1 \\ -  \frac1{\tilde\mu_0 \omega} S_0 ( \beta_0 \partial_\tau  + \beta_1 L_1 )
\end{pmatrix} e^{inc}_3  ,
\end{align*}
and 
$
K_j := \left( NS_j \pm \tfrac12 I\right)^{-1} ND_j , \, L_j : = 2 (TD_j -  TS_j K_j ) .
$
The rest of the densities satisfy $\phi_1 = \phi_0 + e^{inc}_3 ,\, \psi_1 = \psi
_0, \, \eta_j = K_j \phi_j$ and $\xi_j = K_j \psi_j .$

\end{theorem}

The solutions $e_3^{sc}$ and $h_3^{sc} $ of  \eqref{maxwell_third} -- \eqref{radiation} have the asymptotic behavior
\begin{equation}\label{far}
e_3^{sc}  (\bm x ) = \frac{e^{i\kappa_0 r }}{\sqrt{r}} e^\infty (\bm{\hat{x}}) + \mathcal{O} ( r^{-3/2}) , \quad h_3^{sc}  (\bm x ) = \frac{e^{i\kappa_0 r }}{\sqrt{r}} h^\infty (\bm{\hat{x}}) + \mathcal{O} (r^{-3/2}),
\end{equation}
where $\bm{\hat{x}} = \bm x / |\bm x|.$ The pair $(e^\infty, h^\infty)$ is called the far-field pattern corresponding to the scattering problem  \eqref{maxwell_third} -- \eqref{radiation}. Its knowledge is essential for the inverse problem and using \eqref{potentials} we can compute it by
\begin{equation}
\begin{aligned}
e^\infty (\bm{\hat{x}})  &= ( D^\infty \phi_0 ) (\bm{\hat{x}}) - ( S^\infty \eta_0 ) (\bm{\hat{x}}), & \bm{\hat{x}} \in \SQ ,\\
h^\infty (\bm{\hat{x}})  &= ( D^\infty \psi_0 ) (\bm{\hat{x}}) - ( S^\infty \xi_0 ) (\bm{\hat{x}}), & \bm{\hat{x}} \in \SQ ,\\
\end{aligned}
\end{equation}
where $\SQ$ is the unit ball. The far-field operators are given by
\begin{equation}\label{single_double_far}
\begin{aligned}
( S^\infty f) (\bm{\hat{x}}) &= \int_\Gamma \Phi^\infty (\bm{\hat{x}},\bm y) f(\bm y) ds (\bm y), & \bm{\hat{x}} \in \SQ , \\
( D^\infty f) (\bm{\hat{x}}) &= \int_\Gamma \frac{\partial \Phi^\infty}{\partial n (\bm y)} (\bm{\hat{x}},\bm y) f(\bm y) ds (\bm y), & \bm{\hat{x}} \in \SQ ,
\end{aligned}
\end{equation}
where $\Phi^\infty$ is the far-field of the Green function $\Phi ,$ given by
\[
\Phi^\infty (\bm{\hat{x}},\bm y)  = \frac{e^{i\pi /4}}{\sqrt{8\pi \kappa_0}}  e^{-i\kappa_0 \bm{\hat{x}} \cdot \bm y}.
\]

  \section{The inverse problem}\label{inverse}
  
 The inverse scattering problem, we address here, reads: Find the shape and the position of the inclusion $\Omega,$ meaning reconstruct its boundary $\Gamma$, given the far-field patterns $(e^\infty (\bm{\hat{x}}),  h^\infty (\bm{\hat{x}})),$ for all $\bm{\hat{x}} \in \SQ,$ for one or few incident fields \eqref{incident_el}.

  \subsection{The integral equation method} 

To solve the inverse problem we apply the method of nonlinear boundary integral equations, which in our case results to a system of four integral equations on the unknown boundary and one on the unit circle where the far-field data are defined. This method was
 first introduced in \cite{KreRun05} and further considered in various inverse problems, see for instance \cite{CakKre07, CakKreSchu10, ChaGinMin16, GinMin11, IvaJoh08}. Since the direct problem was solved with the direct method (Green's formulas), in order to obtain our numerical data, here we adopt a different approach based on the indirect integral equation method, using simple representations for the fields.  

We assume a double-layer representation for the interior fields and a single-layer representation for the exterior fields. Thus, we set
\begin{equation}\label{fields_inverse}
\begin{aligned}
\cei (\bm x) &= \frac1{\tilde\epsilon_1} (\mathcal D_1 \phi_e )(\bm x),  & \chii (\bm x) &= \frac1{\tilde\mu_1} (\mathcal D_1 \phi_h )(\bm x), & \bm x \in \Omega_1 , \\
e_3^{sc} (\bm x) &= \frac1{\tilde\epsilon_0} (\mathcal S_0 \psi_e )(\bm x),  & h_3^{sc} (\bm x) &= \frac1{\tilde\mu_0} (\mathcal S_0 \psi_h )(\bm x), & \bm x \in \Omega_0 .
\end{aligned}
\end{equation}

Substituting the above representations in the transmission conditions \eqref{boundary_third} and considering the well-known jump relations, we get the system of integral equations
\begin{equation}\label{boundary_inverse}
\begin{aligned}
\frac1{\tilde\epsilon_1} \left(D_1 -\frac12 \right) \phi_e - \frac1{\tilde\epsilon_0} S_0 \psi_e  &= e^{inc}_3 , & \mbox{on } \Gamma,\\
\omega ND_1 \phi_h + \frac{\beta_1}{\tilde\epsilon_1} \left( TD_1 -  \frac12 \frac{\partial}{\partial \tau} \right) \phi_e - \omega\left( NS_0 -\frac12 \right) \psi_h - \frac{\beta_0}{\tilde\epsilon_0} TS_0 \psi_e &= \beta_0 \frac{\partial e^{inc}_3}{\partial \tau }, & \mbox{on } \Gamma,\\
\frac1{\tilde\mu_1} \left(D_1 -\frac12 \right) \phi_h - \frac1{\tilde\mu_0} S_0 \psi_h  &= 0, & \mbox{on } \Gamma,\\
\omega ND_1 \phi_e - \frac{\beta_1}{\tilde\mu_1} \left( TD_1 -  \frac12 \frac{\partial}{\partial \tau} \right) \phi_h - \omega \left( NS_0 -\frac12 \right) \psi_e +  \frac{\beta_0}{\tilde\mu_0} TS_0 \psi_h &= \tilde\epsilon_0 \omega \frac{\partial e^{inc}_3}{\partial n }, & \mbox{on } \Gamma.
\end{aligned}
\end{equation}

In addition, given the far-field operators \eqref{single_double_far} and the representations \eqref{fields_inverse} of the exterior fields we see that the unknown boundary $\Gamma$ and the densities $\psi_e$ and $\psi_h$ satisfy also the far-field equations
\begin{subequations}\label{far_inverse}
\begin{align}
\frac1{\tilde\epsilon_0} S^\infty \psi_e &= e^\infty , \quad \mbox{on } \SQ , \label{far_inverse1} \\
\frac1{\tilde\mu_0} S^\infty \psi_h &= h^\infty , \quad \mbox{on } \SQ , \label{far_inverse2}
\end{align}
\end{subequations}
  where the right-hand sides are the known far-field patterns from the direct problem. The equation \eqref{boundary_inverse} in matrix form reads
  \begin{equation}\label{equation_boundary}
 ( \textbf T + \textbf K ) \bm\phi  = \textbf b ,
  \end{equation}
  where
  \begin{align*}
  \textbf T &= \begin{pmatrix}
  \dfrac{\omega}2 & \dfrac{\beta_1}{2\tilde\mu_1}\partial_\tau & 0 & 0\\[10pt]
  0 & -\dfrac1{2\tilde\mu_1} & 0 & 0 \\
  0 & 0 & \dfrac{\omega}2 & -\dfrac{\beta_1}{2\tilde\epsilon_1}\partial_\tau \\[10pt]
  0 & 0 & 0 & -\dfrac1{2\tilde\epsilon_1}
\end{pmatrix}  , &
  \textbf K &= \begin{pmatrix}
 -\omega NS_0 & - \dfrac{\beta_1}{\tilde\mu_1} TD_1 & \dfrac{\beta_0}{\tilde\mu_0} TS_0 & \omega ND_1 \\[10pt]
 0 & \dfrac1{\tilde\mu_1} D_1 &  -\dfrac1{\tilde\mu_0} S_0 & 0 \\
 -\dfrac{\beta_0}{\tilde\epsilon_0} TS_0 & \omega ND_1 & -\omega NS_0 & \dfrac{\beta_1}{\tilde\epsilon_1} TD_1 \\[10pt]
 -\dfrac1{\tilde\epsilon_0} S_0 & 0 & 0 &  \dfrac1{\tilde\epsilon_1} D_1  
\end{pmatrix} , \\
\bm\phi &= \begin{pmatrix}
\psi_e \\ \phi_h \\ \psi_h \\ \phi_e
\end{pmatrix}, & \textbf b &= \begin{pmatrix}
\tilde\epsilon_0 \omega \partial_n \\ 0 \\ \beta_0 \partial_\tau \\ 1
\end{pmatrix} e^{inc}_3 .
  \end{align*}
  
  The matrix $\textbf T$ due to its special form and the boundness of $\partial_\tau : H^{1/2} (\Gamma) \rightarrow H^{-1/2} (\Gamma)$ has a bounded inverse given by
  \begin{equation}
   \textbf T^{-1} = \begin{pmatrix}
  \dfrac2{\omega} & \dfrac{2\beta_1}{\omega}\partial_\tau & 0 & 0\\[10pt]
  0 & -2\tilde\mu_1 & 0 & 0 \\[6pt]
  0 & 0 & \dfrac2{\omega} & -\dfrac{2\beta_1}{\omega}\partial_\tau \\[10pt]
  0 & 0 & 0 & -2\tilde\epsilon_1
\end{pmatrix} .
  \end{equation}
  
  Then, equation \eqref{equation_boundary} takes the form
  \begin{equation}\label{inverse_final}
  ( \textbf I + \textbf C ) \bm\phi  = \textbf g ,
  \end{equation}
where now $\textbf I$ is the identity matrix and
\begin{align*}
\textbf C &=  \textbf T^{-1} \textbf K = \begin{pmatrix}
-2 NS_0 & 0 & \dfrac{2}{\omega \tilde\mu_0} (\beta_0 - \beta_1) TS_0 & 2ND_1 \\
0 & -2 D_1 & 2 \dfrac{\tilde\mu_1}{\tilde\mu_0} S_0 & 0\\
-\dfrac{2}{\omega \tilde\epsilon_0} (\beta_0 - \beta_1) TS_0 & 2ND_1 & -2NS_0 & 0 \\
2 \dfrac{\tilde\epsilon_1}{\tilde\epsilon_0} S_0 & 0 & 0 & -2D_1
\end{pmatrix}, \\
\textbf g &= \textbf T^{-1} \textbf b = \begin{pmatrix}
2\tilde\epsilon_0  \partial_n \\ 0 \\ \dfrac{2}{\omega}(\beta_0 - \beta_1 ) \partial_\tau \\ -2 \tilde\epsilon_1
\end{pmatrix} e^{inc}_3 .
\end{align*}  

Using the mapping properties of the integral operators \cite{Kre99}, we see that the operator $\textbf C : (H^{-1/2} (\Gamma) \times H^{1/2} (\Gamma))^2  \rightarrow (H^{-3/2} (\Gamma) \times H^{-1/2} (\Gamma))^2 $ is compact.
  
We observe that we have six equations \eqref{inverse_final} and \eqref{far_inverse} for the five unknowns: $\Gamma$ and the four densities. Thus, we consider the linear combination $\tilde\epsilon_0 \cdot$\eqref{far_inverse1} + $\tilde\mu_0 \cdot$\eqref{far_inverse2} as a replacement for the far-field equations in order to state the following theorem as a formulation of the inverse problem.

\begin{theorem}
Given the incident field \eqref{incident_el} and the far-field patterns $(e^\infty (\bm{\hat{x}}),  h^\infty (\bm{\hat{x}})),$ for all $\bm{\hat{x}} \in \SQ,$ if the boundary $\Gamma$ and the densities $\psi_e ,\, \phi_h ,\, \psi_h$ and $\phi_e$ satisfy the system of equations
\begin{subequations}\label{final_system}
\begin{align}
\psi_e-2 NS_0 \psi_e  + \dfrac{2}{\omega \tilde\mu_0} (\beta_0 - \beta_1) TS_0 \psi_h + 2ND_1 \phi_e &= 2\tilde\epsilon_0  \partial_n e^{inc}_3 \label{final_system1},\\
\phi_h -2 D_1 \phi_h + 2 \dfrac{\tilde\mu_1}{\tilde\mu_0} S_0 \psi_h &= 0 \label{final_system2},\\
-\dfrac{2}{\omega \tilde\epsilon_0} (\beta_0 - \beta_1) TS_0 \psi_e + 2ND_1 \phi_h + \psi_h -2NS_0 \psi_h &= \dfrac{2}{\omega}(\beta_0 - \beta_1 ) \partial_\tau e^{inc}_3 \label{final_system3},\\
2 \dfrac{\tilde\epsilon_1}{\tilde\epsilon_0} S_0 \psi_e + \phi_e -2D_1  \phi_e &=  -2 \tilde\epsilon_1 e^{inc}_3 \label{final_system4},\\
S^\infty \psi_e + S^\infty \psi_h&= \tilde\epsilon_0 e^\infty + \tilde\mu_0 h^\infty ,  \label{final_system5}
\end{align}
\end{subequations}
then, $\Gamma$ solves the inverse problem.
\end{theorem}
  
The integral operators in \eqref{final_system} are linear with respect to the densities but non-linear with respect to the unknown boundary $\Gamma.$ The smoothness of the kernels in the far-field equation \eqref{final_system5} reflects the ill-posedness of the inverse problem.

To solve the above system of equations, we consider the method first introduced in \cite{JohSle07} and then applied in different problems, see for instance \cite{AltKre12a, ChaGinMin16, Lee15}. More precisely, given an initial approximation for the boundary $\Gamma$, we solve the subsystem \eqref{final_system1} - \eqref{final_system4} for the densities 
$\psi_e ,\, \phi_h ,\, \psi_h$ and $\phi_e .$ Then, keeping the densities $\psi_e$ and $\psi_h$ fixed we linearize the far-field equation \eqref{final_system5} with respect to the boundary. The linearized equation is solved to obtain the update for the boundary. The linearization is performed using Fr\'echet derivatives of the operators and we also regularize the ill-posed last equation.
  
To present the proposed method in details, we consider the following parametrization for the boundary
\[
\Gamma = \{ \bm z (t) =  r (t) (\cos t,\, \sin t) : t \in [0,2\pi]\},
\]
where $\bm z : \R \rightarrow \R^2$ is a $C^2$-smooth, $2\pi$-periodic, injective in $[0,2\pi),$ meaning that $\bm z'  (t) \neq 0,$ for all $t\in [0,2\pi].$ The non-negative function $r$ represents the radial distance of $\Gamma$ from the origin. Then, we define
\[
\zeta_e (t) = \psi_e (\bm z (t)), \quad   \zeta_h (t) = \psi_h (\bm z (t)), \quad \xi_e  (t) = \phi_e (\bm z (t)), \quad \xi_h  (t) = \phi_h (\bm z (t)), \quad t \in [0,2\pi]
\]
and the parametrized form of \eqref{final_system} is given by
\begin{equation}\label{eqFinal}
\begin{pmatrix}
\bb{A}_1  \\
\bb{A}_2\\
\bb{A}_3 \\
\bb{A}_4\\
\bb{A}_5 \end{pmatrix} (r; \zeta_e) + \begin{pmatrix}
\bb{B}_1  \\
\bb{B}_2\\
\bb{B}_3 \\
\bb{B}_4\\
\bb{B}_5\end{pmatrix} (r; \xi_h ) + \begin{pmatrix}
\bb{C}_1  \\
\bb{C}_2\\
\bb{C}_3 \\
\bb{C}_4\\
\bb{C}_5\end{pmatrix} (r; \zeta_h  ) + \begin{pmatrix}
\bb{D}_1  \\
\bb{D}_2\\
\bb{D}_3 \\
\bb{D}_4\\
\bb{D}_5\end{pmatrix} (r; \xi_e ) = \begin{pmatrix}
\bb{F}_1   \\
\bb{F}_2 \\
\bb{F}_3 \\
\bb{F}_4\\
\bb{F}_5\end{pmatrix} , 
\end{equation}
with the parametrized operators
\begin{align*}
(\bb{A}_1 (r; \zeta ))(t) &= (\bb{C}_3 (r; \zeta ))(t) = \zeta (t) - 2 \int_0^{2\pi} M^{NS_0} (t,s) \zeta (s) ds ,\\
(\bb{A}_3 (r; \zeta ))(t) &= -\frac{\tilde\mu_0}{\tilde\epsilon_0} (\bb{C}_1 (r; \zeta ))(t) = -\frac{2}{\omega \tilde\epsilon_0} (\beta_0 - \beta_1) \int_0^{2\pi} M^{TS_0} (t,s) \zeta (s) ds,\\
(\bb{A}_4 (r; \zeta ))(t) &= \frac{\tilde\mu_0 \tilde\epsilon_1}{\tilde\mu_1 \tilde\epsilon_0}(\bb{C}_2 (r; \zeta ))(t) = 2 \frac{\tilde\epsilon_1}{\tilde\epsilon_0} \int_0^{2\pi} M^{S_0} (t,s) \zeta (s) ds,\\
(\bb{A}_5 (r; \zeta ))(t) &= (\bb{C}_5 (r; \zeta ))(t) = \int_0^{2\pi} \Phi^\infty (\bm{\hat{z}}(t),\bm z (s)) \zeta (s)  |\bm z'  (s)| ds,\\
(\bb{B}_2 (r; \xi ))(t) &= (\bb{D}_4 (r; \xi ))(t) = \xi (t) - 2 \int_0^{2\pi} M^{D_1} (t,s) \xi (s) ds,\\
(\bb{B}_3 (r; \xi ))(t) &= (\bb{D}_1 (r; \xi ))(t) = 2 \int_0^{2\pi} M^{ND_1} (t,s) \xi (s) ds,\\
\end{align*}
and the right-hand side
\begin{align*}
(\bb{F}_1 (r ))(t) &= 2\tilde\epsilon_0  \partial_n e^{inc}_3 (\bm z (t)) , &
(\bb{F}_3 (r ))(t) &=  \dfrac{2}{\omega}(\beta_0 - \beta_1 ) \partial_\tau e^{inc}_3 (\bm z (t)),\\
(\bb{F}_4 (r ))(t) &= -2 \tilde\epsilon_1 e^{inc}_3 (\bm z (t)), &
(\bb{F}_5 )(t) &= \tilde\epsilon_0 e^\infty (\bm{\hat{z}}(t)) + \tilde\mu_0 h^\infty (\bm{\hat{z}}(t)) .
\end{align*}

In addition, we set $\bb{A}_2 = \bb{B}_1 = \bb{B}_4 = \bb{B}_5 = \bb{C}_4 = \bb{D}_2 = \bb{D}_3 = \bb{D}_5 = \bb{F}_2 = 0.$ The matrix $M^{K_j}$ denotes the discretized kernel of the operator $K_j.$ The explicit forms of the kernels can be found for example in \cite[Equation 4.3]{GinMin16}. The operators $\bb{A}_k ,\, \bb{B}_k ,\, \bb{C}_k ,\, \bb{D}_k,$ $k=1,2,3,4,5$ act on the densities and the first variable $r$ shows the dependence on the unknown parametrization of the boundary. Only $\bb{F}_5$ is independent of the radial function. 

Let the function $q$ stand for the radial function of the perturbed boundary
\[
\Gamma_q = \{ \bm q (t) =  q (t) (\cos t,\, \sin t) : t \in [0,2\pi]\}.
\]
Then the iterative method reads:
\begin{iterative}\label{IterationScheme}
Let $r^{(0)}$ be an initial approximation of the radial function. Then, in the $k$th iteration step:
\begin{enumerate}[i.]
\item We assume that we know $r^{(k-1)}$ and we solve the subsystem
\begin{equation}\label{eqFinal2}
\begin{pmatrix}
\bb{A}_1  \\
\bb{A}_2\\
\bb{A}_3 \\
\bb{A}_4 \end{pmatrix} (r^{(k-1)}; \zeta_e) + \begin{pmatrix}
\bb{B}_1  \\
\bb{B}_2\\
\bb{B}_3 \\
\bb{B}_4\end{pmatrix} (r^{(k-1)}; \xi_h ) + \begin{pmatrix}
\bb{C}_1  \\
\bb{C}_2\\
\bb{C}_3 \\
\bb{C}_4\end{pmatrix} (r^{(k-1)}; \zeta_h  ) + \begin{pmatrix}
\bb{D}_1  \\
\bb{D}_2\\
\bb{D}_3 \\
\bb{D}_4\end{pmatrix} (r^{(k-1)}; \xi_e ) = \begin{pmatrix}
\bb{F}_1   \\
\bb{F}_2 \\
\bb{F}_3 \\
\bb{F}_4\end{pmatrix} , 
\end{equation}
to obtain the densities $\zeta_e^{(k)}, \, \xi_h^{(k)}, \,\zeta_h^{(k)}$ and $\xi_e^{(k)}.$
\item Keeping the densities $\zeta_e$ and $\zeta_h$ fixed, we linearize the fifth equation of \eqref{eqFinal}, namely
\begin{equation}\label{linear_eq}
\bb A_5 (r^{(k-1)}; \zeta_e^{(k)}) + (\bb A'_5 (r^{(k-1)}; \zeta_e^{(k)})) (q) + \bb C_5 (r^{(k-1)}; \zeta_h^{(k)}) + (\bb C'_5 (r^{(k-1)}; \zeta_h^{(k)})) (q) = \bb F_5. 
\end{equation}
We solve this equation for $q$ and we update the radial function $r^{(k)} = r^{(k-1)} +q.$
\end{enumerate}
The iteration stops when a suitable stopping criterion is satisfied.
\end{iterative}

\begin{remark}
In order to take advantage of the available measurement data, we can also keep the overdetermined system \eqref{boundary_inverse} and \eqref{far_inverse}
instead of \eqref{final_system5} and replace equation \eqref{linear_eq} with
\begin{equation}\label{linear_eq3}
\begin{pmatrix}
\bb A'_5 (r^{(k-1)}; \zeta_e^{(k)}) \\
\bb A'_5 (r^{(k-1)}; \zeta_h^{(k)})
\end{pmatrix} q = \begin{pmatrix}
\bb F_e \\
\bb F_h
\end{pmatrix} - \begin{pmatrix}
\bb A_5 (r^{(k-1)}; \zeta_e^{(k)}) \\
\bb A_5 (r^{(k-1)}; \zeta_h^{(k)})
\end{pmatrix},
\end{equation}
where now $\bb F_e  = \tilde\epsilon_0 \, e^\infty$ and $\bb F_h  = \tilde\mu_0 \, h^\infty.$
\end{remark}

The Fr\'echet derivatives of the operators are calculated by formally differentiating their kernels with respect to $r$
\begin{equation}\label{frechet}
((\bb A'_5 (r; \zeta)) (q))(t) = \frac{e^{i\pi /4}}{\sqrt{8\pi \kappa_0}} \int_0^{2\pi} e^{-i\kappa_0 \bm{\hat{z}}(t) \cdot \bm z(s)} \left( -i\kappa_0 \bm{\hat{z}}(t) \cdot \bm q(s) |\bm z'  (s)| + \frac{\bm z'  (s) \cdot \bm q'  (s)}{|\bm z'  (s)|} \right) \zeta (s) ds.
\end{equation}
Recall that $\bb A_5 =\bb C_5 = S^\infty.$ If $\kappa_0^2$ is not an interior Neumann eigenvalue, then the operator $\bb A'_5$ is injective \cite{IvaJoh08}. Using similar arguments as in \cite{AltKre12a, HohScho98} we can relate the above iterative scheme to the classical Newton's method.
    
The iterative scheme \ref{IterationScheme} can also be generalized to the case of multiple illuminations $\ein_l , \, l = 1,...,L .$
\begin{iterative}[Multiple illuminations]
Let $r^{(0)}$ be an initial approximation of the radial function. Then, in the $k$th iteration step:
\begin{enumerate}[i.]
\item We assume that we know $r^{(k-1)}$ and we solve the $L$ subsystems
\begin{multline}\label{eqFinal2multi}
\begin{pmatrix}
\bb{A}_1  \\
\bb{A}_2\\
\bb{A}_3 \\
\bb{A}_4 \end{pmatrix} (r^{(k-1)}; \zeta_{e,l}) + \begin{pmatrix}
\bb{B}_1  \\
\bb{B}_2\\
\bb{B}_3 \\
\bb{B}_4\end{pmatrix} (r^{(k-1)}; \xi_{h,l} ) + \begin{pmatrix}
\bb{C}_1  \\
\bb{C}_2\\
\bb{C}_3 \\
\bb{C}_4\end{pmatrix} (r^{(k-1)}; \zeta_{h,l}  ) \\ + \begin{pmatrix}
\bb{D}_1  \\
\bb{D}_2\\
\bb{D}_3 \\
\bb{D}_4\end{pmatrix} (r^{(k-1)}; \xi_{e,l} ) = \begin{pmatrix}
\bb{F}_{1,l}   \\
\bb{F}_{2,l} \\
\bb{F}_{3,l} \\
\bb{F}_{4,l}
\end{pmatrix} ,  \quad l=1,...,L
\end{multline}
to obtain the densities $\zeta_{e,l}^{(k)}, \, \xi_{h,l}^{(k)}, \,\zeta_{h,l}^{(k)}$ and $\xi_{e,l}^{(k)}.$
\item Then, keeping the densities fixed, we solve the overdetermined version of the linearized fifth equation of \eqref{eqFinal}
\begin{equation*}
\begin{pmatrix}
\bb A'_5 (r^{(k-1)}; \zeta_{e,1}^{(k)}+  \zeta_{h,1}^{(k)}) \\
\bb A'_5 (r^{(k-1)}; \zeta_{e,2}^{(k)}+ \zeta_{h,2}^{(k)}) \\
\vdots \\
\bb A'_5 (r^{(k-1)}; \zeta_{e,l}^{(k)}+ \zeta_{h,l}^{(k)})
\end{pmatrix}  q = \begin{pmatrix}
\bb F_{5,1} - \bb A_5 (r^{(k-1)}; \zeta_{e,1}^{(k)} +\zeta_{h,1}^{(k)}) \\
\bb F_{5,2} - \bb A_5 (r^{(k-1)}; \zeta_{e,2}^{(k)}+ \zeta_{h,2}^{(k)}) \\
\vdots \\
\bb F_{5,L} - \bb A_5 (r^{(k-1)}; \zeta_{e,l}^{(k)} + \zeta_{h,l}^{(k)})
\end{pmatrix}
\end{equation*}
for $q$ and we update the radial function $r^{(k)} = r^{(k-1)} +q.$
\end{enumerate}
The iteration stops when a suitable stopping criterion is satisfied.
\end{iterative}  
  
  \section{Numerical implementation}\label{numerics}
  
  In this section, we present numerical examples that illustrate the applicability of the proposed method. We use quadrature rules for integrating the singularities considering trigonometric interpolation. The convergence and error analyses are given in \cite{Kre95, Kre14}. Then, the system of integral equations is solved using the Nystr\"om method. The parametrized forms of the integral operators are presented in \cite[Section 4]{GinMin16}. We approximate the smooth kernels with the trapezoidal rule and the singular ones with the well-known quadratures rules \cite{Kre14}.
  
In the following examples, we consider two different boundary curves. A peanut-shaped and an apple-shaped boundary with radial function 
\[
r (t) =  (0.5 \cos^2 t + 0.15 \sin^2 t)^{1/2}, \quad t \in [0,2\pi], 
\]
and
\[
r (t) =  \frac{0.45 + 0.3 \cos t -0.1 \sin 2t}{1+0.7 \cos t}, \quad t \in [0,2\pi],
\]
respectively. 

To avoid an inverse crime, we construct the simulated far-field data using the numerical scheme \eqref{eq_general1} and considering double amount of quadrature points compared to the inverse problem. We approximate the radial function $q$ by a trigonometric polynomial of the form
\[
q (t) \approx \sum_{k=0}^m a_k \cos kt + \sum_{k=1}^m b_k \sin kt , \quad t\in[0,2\pi],
\]
and we consider $2n$ equidistant points $t_j = j\pi/n, \, j=0,...,2n-1 .$ The well-posed subsystem \eqref{eqFinal2} does not require any special treatment. The ill-posed linearized far-field equation \eqref{linear_eq} is solved by Tikhonov regularization. We rewrite \eqref{linear_eq} as
\begin{equation}\label{linear_eq2}
 (\bb A'_5 (r^{(k-1)}; \zeta^{(k)})) (q)= \bb F_5 - \bb A_5 (r^{(k-1)}; \zeta^{(k)}),
\end{equation}
for $\zeta^{(k)} := \zeta_e^{(k)} +\zeta_h^{(k)},$ and we decompose \eqref{frechet} as
\begin{equation}\label{frechet2}
((\bb A'_5 (r; \zeta)) (q))(t) = ((\bb G_1 (r; \zeta)) (q))(t) + ((\bb G_2 (r; \zeta)) (q'))(t) ,
\end{equation}
where
\begin{align*}
((\bb G_1 (r; \zeta)) (q))(t) &:= \frac{e^{i\pi /4}}{\sqrt{8\pi \kappa_0}} \int_0^{2\pi} e^{-i\kappa_0 \bm{\hat{z}}(t) \cdot \bm z(s)} \left[ -i\kappa_0 \bm{\hat{z}}(t) \cdot (\cos s,\, \sin s) |\bm z'  (s)| \right. \\
 &\phantom{=}\left. + \frac{\bm z'  (s) \cdot (-\sin s,\, \cos s)}{|\bm z'  (s)|} \right] \zeta (s) q(s) ds , \\
((\bb G_2 (r; \zeta)) (q'))(t) &:= \frac{e^{i\pi /4}}{\sqrt{8\pi \kappa_0}} \int_0^{2\pi} e^{-i\kappa_0 \bm{\hat{z}}(t) \cdot \bm z(s)}  \frac{\bm z'  (s) \cdot  (\cos s,\, \sin s)}{|\bm z'  (s)|}  \zeta (s) q' (s) ds .
\end{align*}

We replace the derivative of $q$ by the derivative of the trigonometric interpolation polynomial
\[
q' (t) \approx  \sum_{j=0}^{2n-1} \textbf Q (t ,t_j ) q (t_j),
\]
with weight
\[
\textbf Q (t_k , t_j) = \frac12 (-1)^{k-j}  \cot \frac{t_k -t_j}2 , \quad k \neq j, \,  k = 0,...,2n-1 .
\]

Then, at the $k$th step we minimize the Tikhonov functional of the discretized equation
\[
\|\textbf A \textbf T \textbf x - \textbf b\|^2_2 + \lambda \|\textbf x\|_p^p , \quad \lambda >0 ,
\]
where $\textbf x \in \R^{(2m+1)\times 1}$ is the vector with the unknowns coefficients $a_0 ,...,a_m,$ $b_1 ,..., b_m$ of the radial function, and $\textbf A \in \C^{2n \times 2n}, \, \textbf b \in \C^{2n \times 1}$ are given by
\begin{align*}
\textbf A_{kj} &= \textbf M^{\bb G_1} (t_k , t_j) + \textbf M^{\bb G_2} (t_k , t_j) \textbf Q (t_k , t_j),  \\
\textbf b_k &= \bb F_5 (t_k) - ( \textbf M^{\bb A_5} \bm \zeta )(t_k ) ,
\end{align*}
for $k,j = 0,...,2n-1.$ The multiplication matrix $\textbf T\in \R^{2n\times (2m+1)}$ stands for the trigonometric functions of the approximated radial function and is given by
\begin{equation*}
\textbf T_{kj} =
\begin{cases}
\cos \frac{kj\pi}n , & k = 0,...,2n-1, \, j = 0,...,m \\
\sin \frac{k(j-m)\pi}n ,  & k = 0,...,2n-1, \, j = m+1,...,2m
\end{cases}
\end{equation*}
Here $p\geq 0$ defines the corresponding Sobolev norm. Since $q$ is real valued we solve the following regularized equation
\begin{equation}\label{tikhonov}
\left(
\textbf T^\top \left(  \RE (\textbf A)^\top  \RE (\textbf A) +  \IM (\textbf A)^\top \IM (\textbf A)  \right) \textbf T + \lambda_k \textbf I_p \right) \textbf x  = \textbf T^\top \left(  \RE (\textbf A)^\top  \RE (\textbf b) +  \IM (\textbf A)^\top \IM (\textbf b)  \right),
\end{equation}
on the $k$th step, where the matrix $\textbf I_p \in \R^{(2m+1)\times (2m+1)}$ corresponds to the Sobolev $H^p$ penalty term. We solve \eqref{tikhonov} using the conjugate gradient method.
We update the regularization parameter in each iteration step $k$ by 
\[
\lambda_k  = \lambda_0 \left( \frac23\right)^{k-1}, \quad k=1,2,...
\]
for some given initial parameter $\lambda_0 >0.$ To test the stability of the iterative method against noisy data, we add also noise to the far-field patterns with respect to the $L^2-$norm
\[
e^\infty_\delta = e^\infty + \delta_1 \frac{\|e^\infty \|_2}{\|u\|_2} u ,  \quad 
h^\infty_\delta = h^\infty + \delta_2 \frac{\|h^\infty \|_2}{\|v\|_2} v ,
\]
for some given noise levels $\delta_1 , \, \delta_2$ where $ u = u_1 +\i u_2 , \,  v = v_1 +\i v_2,$ for $u_1 , u_2 , v_1, v_2 \in \R$ normally distributed random variables.

\begin{figure}[t]
\begin{center}
\includegraphics[scale=0.7]{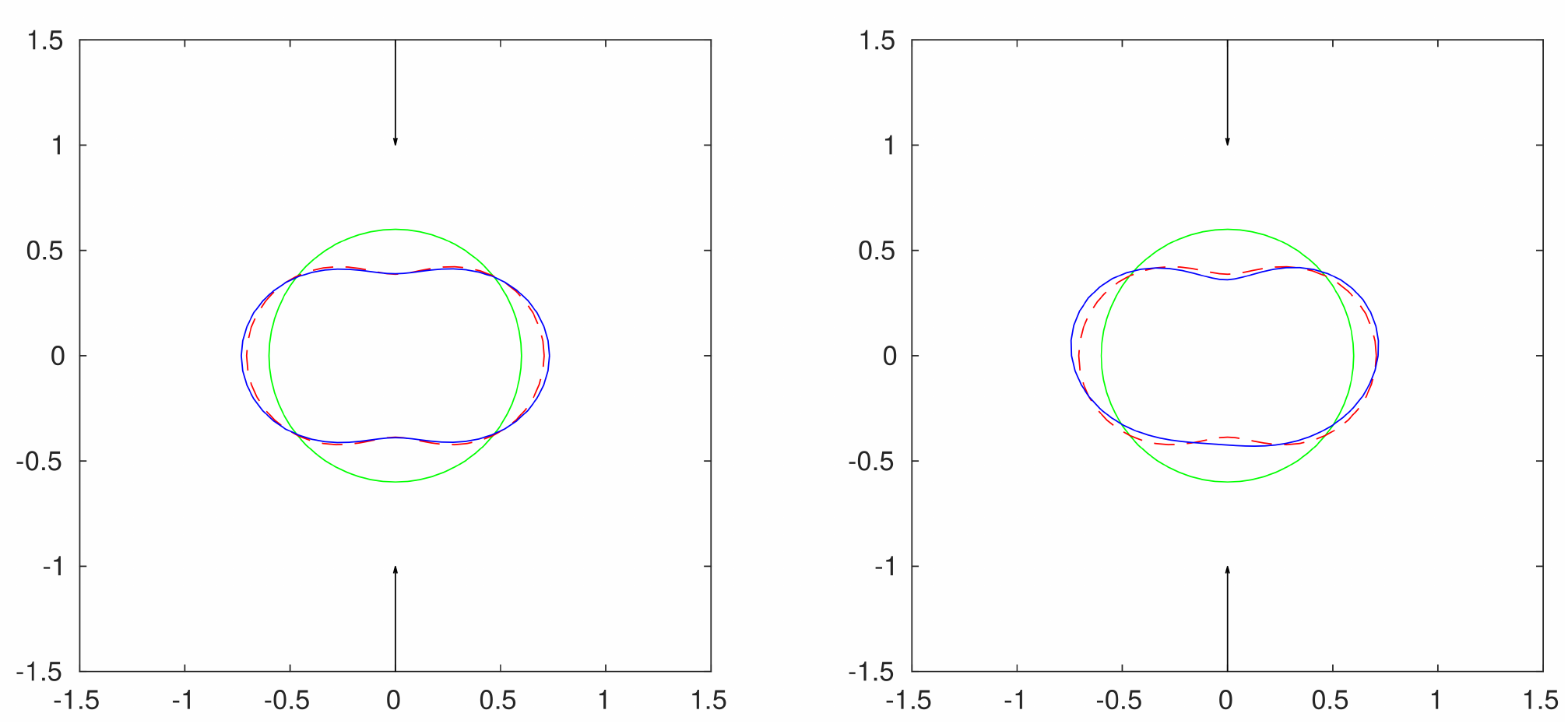}
\caption{Reconstruction of a peanut-shaped boundary for exact (left) and noisy (right) data. }\label{Fig1b}
\end{center}
\end{figure}

Already in simpler cases \cite{AltKre12a}, the knowledge of the far-field patterns for one incident wave is not enough to produce satisfactory reconstructions. Thus, we will also use multiple incident directions. To do so, we have to consider different values of the polar angle $\phi$ since in $\R^2 ,$ as we see from \eqref{incident_el}, corresponds to the incident direction. We set
\[
\textbf d_l = \left(  \cos \frac{2\pi l}L , \, \sin \frac{2\pi l}L \right), \quad l=1,...,L.
\]

\subsection{Numerical results}

In all examples we set the exterior parameters $(\epsilon_0 , \mu_0) = (1,1)$ and the interior $(\epsilon_1 , \mu_1) = (2,2).$ We use $n=64$ collocation points for the direct problem and $n=32$ for the inverse. We set $\theta = \pi /3$ and $\lambda_0  \in [0.5, \, 0.8]$ as the initial regularization parameter.

We present reconstructions for different boundary curves, different number of incident directions and initial guesses for exact and perturbed far-field data. In all figures the initial guess is a circle with radius $r_0 ,$ a green solid line, the exact curve is represented by a dashed red line and the reconstructed by a solid blue line. The arrows denote the directions of the incoming incident fields.

\begin{figure}[t!]
  \centering
  \includegraphics[scale=0.7]{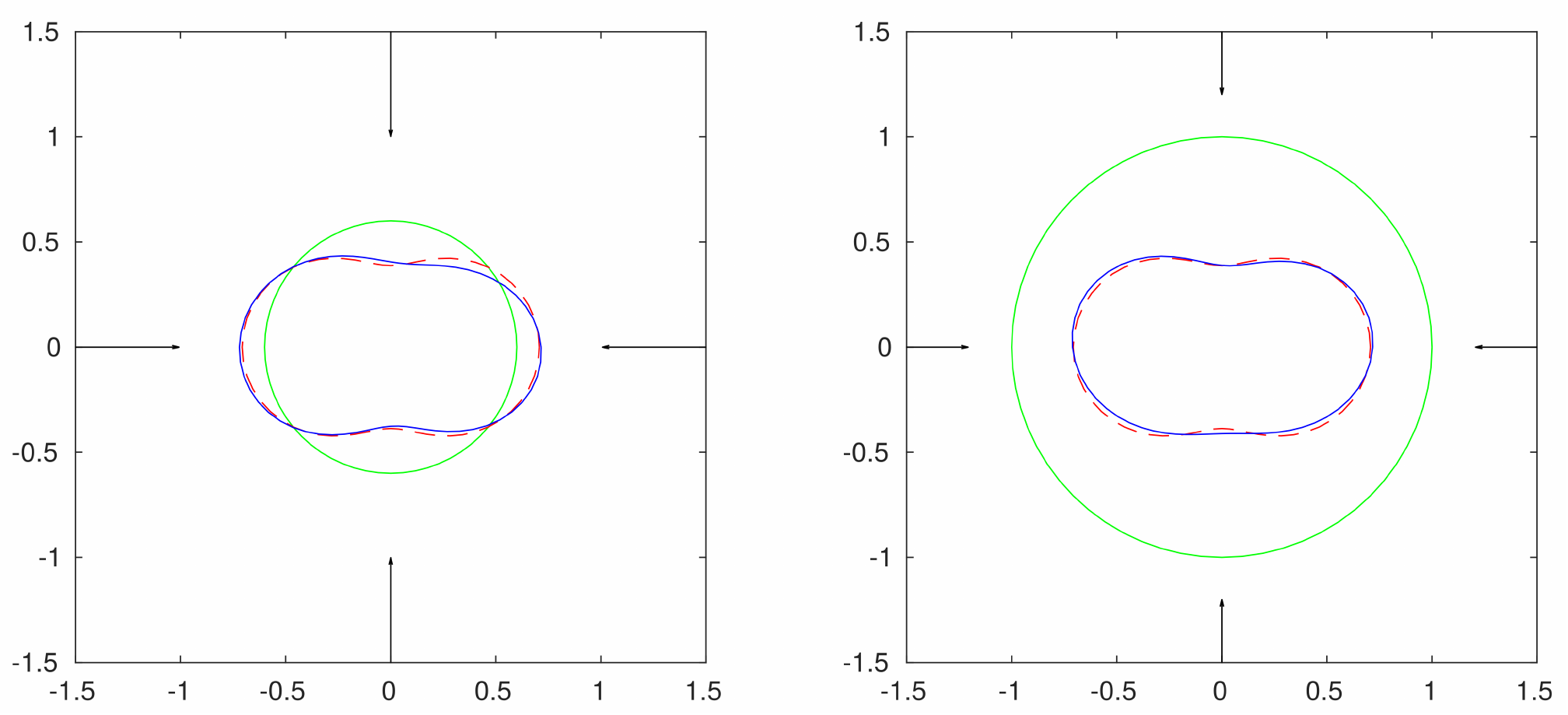}
\caption{Reconstruction of a peanut-shaped boundary for noisy data and different initial guesses. }\label{Fig2}

  \vspace*{\floatsep}

  \includegraphics[scale=0.7]{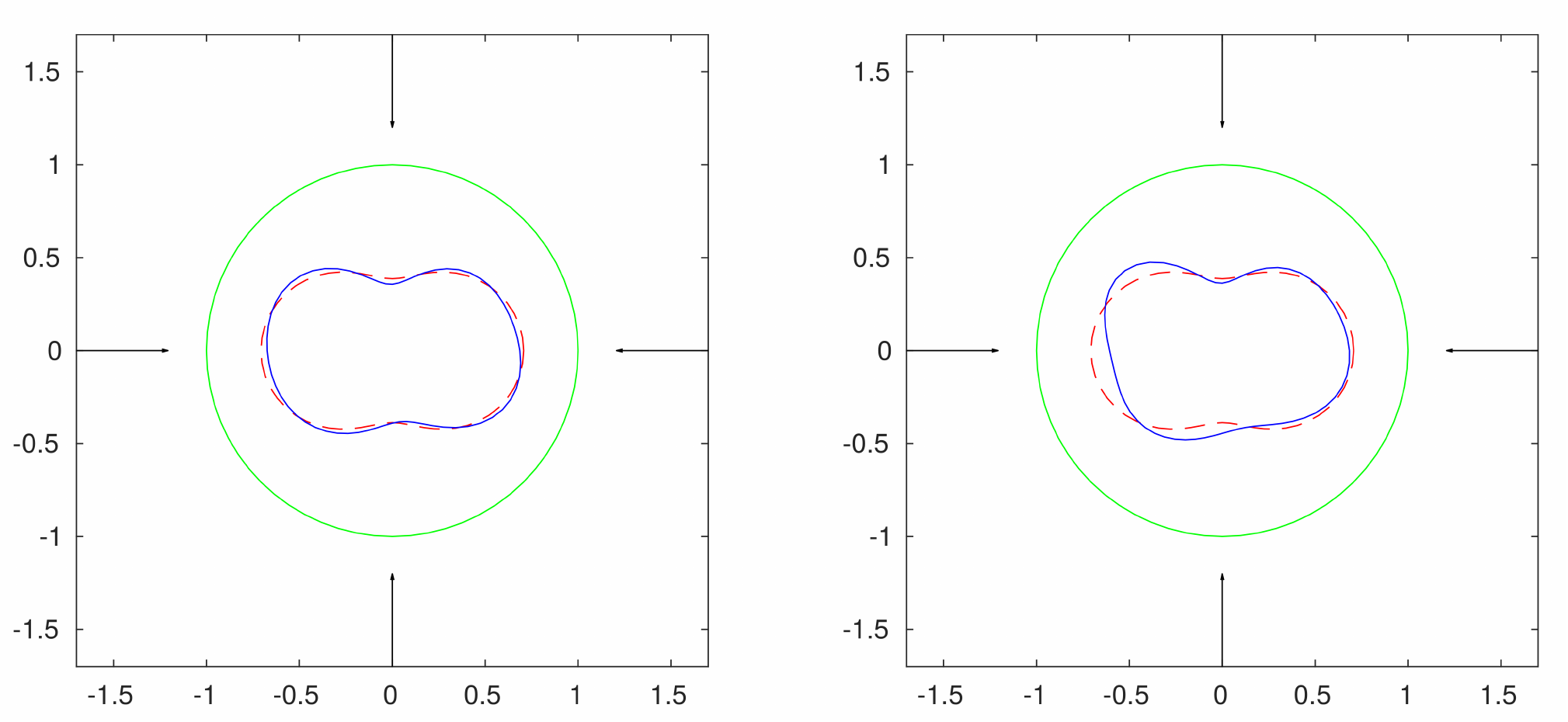}
\caption{Reconstruction of a peanut-shaped boundary for exact (left) and noisy (right) data. }\label{Fig3}
\end{figure}

In the first three examples we consider the peanut-shaped boundary. In the first example, the regularized equation \eqref{tikhonov} is solved with $L^2$ penalty term, meaning $p=0$ and $m=3$ coefficients. We solve equation \eqref{linear_eq3} for different incident directions. The reconstructions for $\omega = 2.5$ and $r_0 = 0.6$ are presented in Figure~\ref{Fig1b} for two incident fields with directions $\textbf d_{l+1/2}.$ On the left picture, we see the reconstructed curve for exact data and 9 iterations and on the right picture for noisy data with  $\delta_1 = \delta_2 = 5
\%$ and 14 iterations. 
In the second example, we consider equation \eqref{linear_eq}, four incident fields, noisy data $\delta_1 = \delta_2 = 5
\%$ and we keep all the parameters as before. The reconstructions for 
$r_0 = 0.6$ and 14 iterations are shown in the left picture of Figure~\ref{Fig2}, and for $r_0 = 1$ and 20 iterations in the right one.
We set $m=5$ and $p=1$ ($H^1$ penalty term) in the third example. The results for $r_0 = 1$ and four incident fields are shown in Figure~\ref{Fig3}. Here $\omega =2$ and we use equation \eqref{linear_eq3}. We need 26 iterations for the exact data and 30 iterations for the noisy data ($\delta_1 = \delta_2 = 3
\%$).

\begin{figure}[t!]
  \centering
  \includegraphics[scale=0.7]{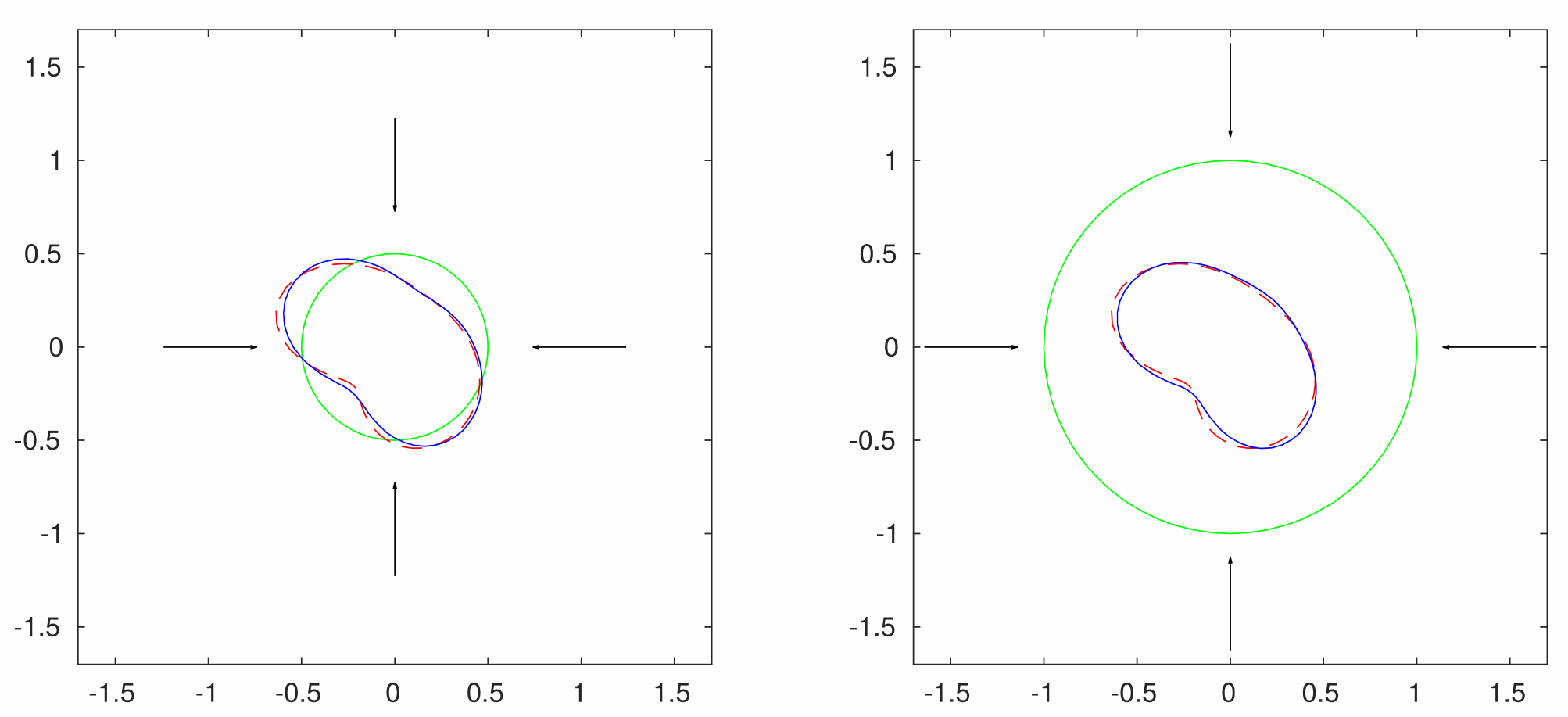}
\caption{Reconstruction of an apple-shaped boundary for exact data and different initial guesses. }\label{Fig4}

  \vspace*{\floatsep}

  \includegraphics[scale=0.7]{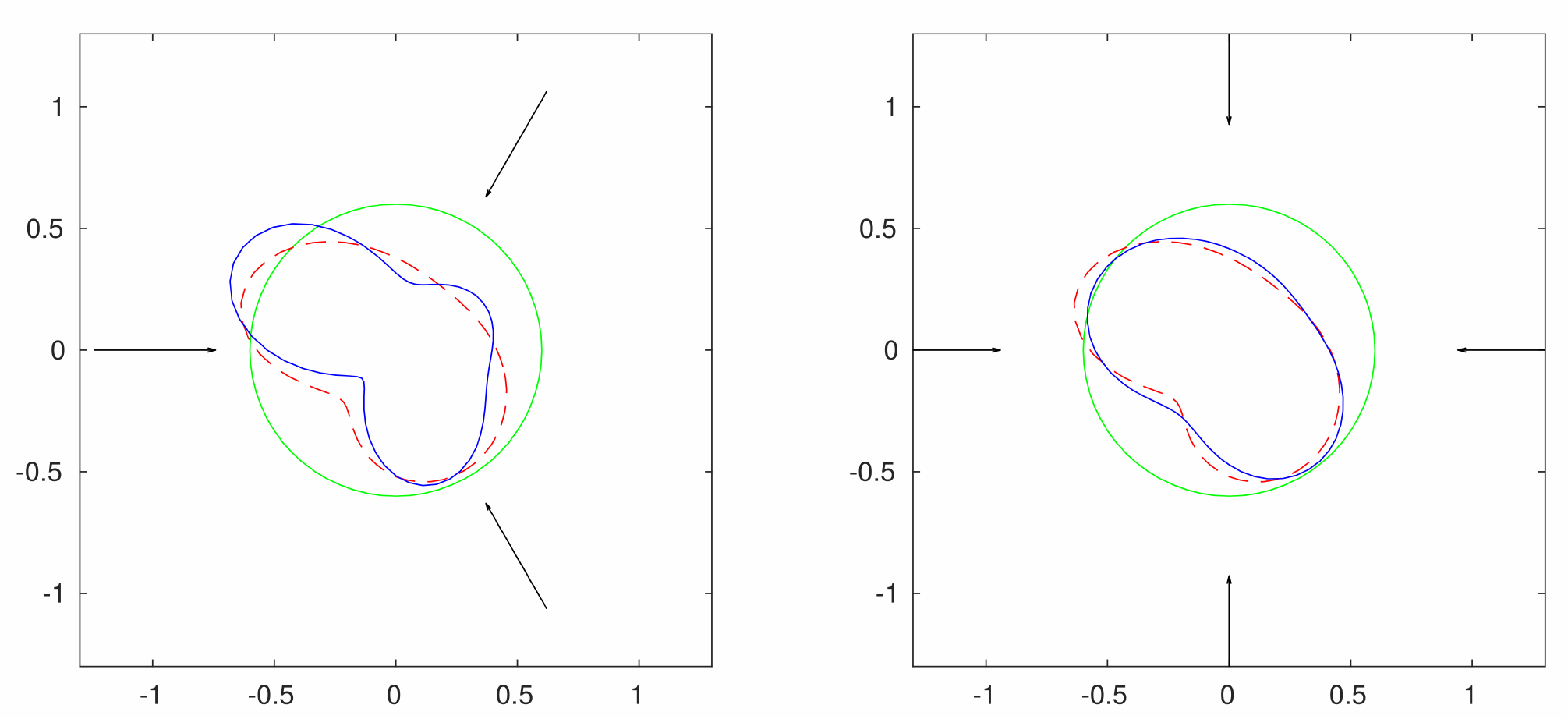}
\caption{Reconstruction of an apple-shaped boundary for noisy data and different number of initial illuminations. }\label{Fig5}
\end{figure}

In the last two examples we consider the apple-shaped boundary, $H^1$ penalty term, $\omega = 3$ and $m=3$ coefficients. In the fourth example, we consider equation \eqref{linear_eq}, noise-free data and four incident fields in order to examine the dependence of the iterative scheme on the initial radial guess. On the left picture of Figure~\ref{Fig4}, we see the reconstructed curve for $r_0 = 0.5$ after 13 iterations and on the right picture for $r_0 = 1$ after 20 iterations. In the last example we consider $\delta_1 = \delta_2 = 3
\%$ noise and $r_0 = 0.6.$ Figure~\ref{Fig5} shows the improvement of the reconstruction for more incident fields. On the left picture we see the results for three incident fields, equation \eqref{linear_eq3}  and 7 iterations and the reconstructed curve for 4 incident fields, equation \eqref{linear_eq}  and 15 iterations is shown on the right picture.

Our examples show the feasibility of the proposed iterative scheme and the stability against noisy data. Considering more than one incident field improves considerably the reconstructions. The choice of the initial guess is also crucial.



\begin{thebibliography}{10}

\bibitem{AltKre12a}
 A. Altundag and R. Kress,  \textit{On a two-dimensional inverse scattering problem for a dielectric,}  Appl. Analysis 91(4)  (2012), pp. 757--771.

\bibitem{CakKre07}
 F. Cakoni and R. Kress,
 \textit{Integral equations for inverse problems in corrosion detection from partial cauchy data},   Inverse Probl. Imaging 1(2) (2007), pp. 229--245.

\bibitem{CakKreSchu10}
F. Cakoni, R. Kress  and C. Schuft,  \textit{Integral equations for shape and impedance reconstruction in
  corrosion detection},  Inverse Probl. 26 (2010), pp. 095012.

\bibitem{CanLee91}
A.C. Cangellaris and R. Lee,  \textit{Finite element analysis of electromagnetic scattering from
  inhomogeneous cylinders at oblique incidence}, IEEE Trans. Ant. Prop. 39  (1991), pp. 645--650.

\bibitem{ChaGinMin16}
R. Chapko, D. Gintides  and L. Mindrinos,  \textit{The inverse scattering problem by an elastic inclusion},   submitted  (Preprint on
  ArXiv:1610.07376) (2016), pp. 19.

\bibitem{ColKre13}
D. Colton and R. Kress, \textit{ Integral equation methods in scattering theory}, Classics in Applied Mathematics, Society for Industrial and Applied  Mathematics, Philadelphia, 2013.

\bibitem{ColKre14}
D. Colton  and R. Kress, \textit{Inverse acoustic and electromagnetic scattering theory}, 3~ed.,
  vol.~93 of  Applied Mathematical Sciences,
 Springer-Verlag, New York, 2013.

\bibitem{EckKre07}
H. Eckel  and R. Kress, \textit{Nonlinear integral equations for the inverse electrical impedance
  problem},  Inverse Probl. 23(2) (2007), pp. 475.

\bibitem{GinMin11}
D. Gintides and L. Midrinos,  \textit{Inverse scattering problem for a rigid scatterer or a cavity in
  elastodynamics},  ZAMM Z. Angew. Math. Mech. 91(4) (2011), pp. 276--287.

\bibitem{GinMin16}
D. Gintides  and L. Mindrinos,  \textit{The direct scattering problem of obliquely incident electromagnetic
  waves by a penetrable homogeneous cylinder}, J. Integral Equations Appl. 28(1) (2016), pp. 91--122.

\bibitem{HohScho98}
T. Hohage and C. Schormann,  \textit{A newton-type method for a transmission problem in inverse
  scattering},  Inverse Probl. 14(5) (1998), pp. 1207.

\bibitem{IvaJoh07}
O. Ivanyshyn and B.T. Johansson,  \textit{Nonlinear integral equation methods for the reconstruction of an
  acoustically sound-soft obstacle}, J. Integral Equations Appl. 19(3) (2007), pp. 289--308.

\bibitem{IvaJoh08}
O. Ivanyshyn and B.T. Johansson,  \textit{Boundary integral equations for acoustical inverse sound-soft
  scattering}, J. Inv. Ill-posed Problems 16(1) (2008), pp. 65--78.

\bibitem{JohSle07}
B.T. Johansson and B.D. Sleeman,  \textit{Reconstruction of an acoustically sound-soft obstacle from one
  incident field and the far-field pattern},  IMA J. Appl. Math. 72 (2007), pp. 96--112.

\bibitem{Kre95}
R. Kress,  On the numerical solution of a hypersingular integral equation in
  scattering theory, J. Comput. Appl. Math. 61(3) (1995), pp. 345--360.

\bibitem{Kre14}
R. Kress,  \textit{A collocation method for a hypersingular boundary integral equation
  via trigonometric differentiation}, J. Integral Equations Appl. 26(2) (2014), pp. 197--213.

\bibitem{Kre99}
R, Kress,  \textit{Linear Integral Equations}, 3~ed.  Springer, New York, 2014.

\bibitem{KreRun05}
R. Kress  and W. Rundell,  \textit{Nonlinear integral equations and the iterative solution for an
  inverse boundary value problem}, Inverse Probl. 21 (2005), pp. 1207--1223.

\bibitem{Lee15}
K.M. Lee,  \textit{Inverse scattering problem from an impedance crack via a composite
  method}, Wave Motion 56 (2015), pp. 43--51.

\bibitem{LucPanSche10}
M. Lucido, G.  Panariello  and F. Schettiho, \textit{
Scattering by polygonal cross-section dielectric cylinders at oblique
  incidence},  IEEE Trans. Ant. Prop. 58 (2010), pp. 540--551.

\bibitem{NakSleWan12}
G. Nakamura, B.D. Sleeman and H. Wang,  \textit{On uniqueness of an inverse problem in electromagnetic obstacle
  scattering for an impedance cylinder}, Inverse Probl. 28(5) (2012), pp. 055012.

\bibitem{NakWan13}
G. Nakamura and H. Wang,  \textit{The direct electromagnetic scattering problem from an imperfectly
  conducting cylinder at oblique incidence}, 
 J. Math. Anal. Appl. 397 (2013), pp. 142--155.

\bibitem{TsiAliAnaKak07}
N.L. Tsitsas, E.G. Alivizatos, H.T. Anastassiu and D.I. Kaklamani,
\textit{ Optimization of the method of auxiliary sources (mas) for oblique  incidence scattering by an infinite dielectric cylinder},  Electrical Engineering 89 (2007), pp. 353--361.

\bibitem{WanNak12}
H.  Wang  and G. Nakamura,  \textit{The integral equation method for electromagnetic scattering problem
  at oblique incidence},  Appl. Num. Math. 62 (2012), pp. 860--873.

\end{thebibliography}
\end{document}